\documentclass[11pt]{amsart}

\usepackage{amscd}
\usepackage{amsmath, amssymb}
\usepackage{amsfonts}
\newcommand{\de}{\partial}
\newcommand{\db}{\overline{\partial}}

\newcommand{\ddbar}{\sqrt{-1} \partial \overline{\partial}}

\newcommand{\mn}{\sqrt{-1}}

\newcommand{\tr}[2]{\mathrm{tr}_{#1}{#2}}
\newcommand{\ti}[1]{\tilde{#1}}
\newcommand{\vp}{\varphi}
\newcommand{\vol}{\mathrm{Vol}}

\newcommand{\ve}{\varepsilon}

\renewcommand{\leq}{\leqslant}
\renewcommand{\geq}{\geqslant}
\renewcommand{\le}{\leqslant}

\numberwithin{equation}{section}

\begin{document}
\newtheorem{claim}{Claim}
\newtheorem{theorem}{Theorem}[section]
\newtheorem{lemma}[theorem]{Lemma}
\newtheorem{corollary}[theorem]{Corollary}
\newtheorem{proposition}[theorem]{Proposition}
\newtheorem{question}{question}[section]
\newtheorem{conjecture}[theorem]{Conjecture}

\theoremstyle{definition}
\newtheorem{remark}[theorem]{Remark}
\author{Valentino Tosatti}
\title{The Calabi-Yau Theorem and K\"ahler currents}

\thanks{Supported in part by a Sloan Research Fellowship and NSF grant DMS-1308988.}
 \address{Department of Mathematics, Northwestern University, 2033 Sheridan Road, Evanston, IL 60201}
  \email{tosatti@math.northwestern.edu}
\dedicatory{Dedicated to Professor S.-T. Yau on the occasion of his 65th birthday.}
\begin{abstract}
In this note we give an overview of some applications of the Calabi-Yau theorem to the construction of singular positive $(1,1)$ currents on compact complex manifolds. We show how recent developments allow us to give streamlined proofs of existing results, as well as new ones.
\end{abstract}
\maketitle

\section{Introduction}
The Calabi Conjecture \cite{Ca}, solved by S.-T. Yau in 1976 \cite{Ya}, says that on a compact K\"ahler manifold one can find K\"ahler metrics with prescribed Ricci form, equal to any given representative of the first Chern class of the manifold, and the solution is unique once we fix its cohomology class. Equivalently, one can find K\"ahler metrics with prescribed volume form and, when recast in this way, the Calabi-Yau Theorem is equivalent to saying that given a K\"ahler metric $\omega$ and a smooth function $F$ with $\int_X(e^F-1)\omega^n=0$, one can solve the complex Monge-Amp\`ere equation
\begin{equation}\label{ma}
(\omega+\ddbar\vp)^n=e^F\omega^n,\quad \omega+\ddbar\vp>0,\quad\sup_X\vp=0,
\end{equation}
and obtain a unique smooth solution $\vp$. This result is one of the cornerstones of complex geometry, and it has found a myriad of applications, see e.g. \cite{PSS}. In this note we discuss some of these applications to the construction problem of singular closed positive currents on K\"ahler manifolds, where the basic technique is the ``mass concentration'' procedure of Demailly \cite{De}. Roughly speaking, one solves a family of complex Monge-Amp\`ere equations of the form \eqref{ma} where the right-hand side function $F$ degenerates in the limit, and in certain cases one can prove that a sequence of solutions converges to the desired singular current. Constructing such currents is of great importance for applications to algebraic geometry, see e.g. \cite{Bo, BDPP, CT, De, DP} and references therein.

Very recently a new point of view on the mass concentration technique was discovered by Chiose \cite{Ch}, which greatly simplifies the picture. This was further explored by Xiao \cite{Xi} and Popovici \cite{Po}. In this note we observe that these arguments can be easily modified to provide effective estimates. For example, using the technique of Popovici \cite{Po}, in Section \ref{sect1} we prove the following:
\begin{theorem}\label{main}
Let $(X^n,\omega)$ be a compact K\"ahler manifold, and $\alpha,\beta$ two closed real $(1,1)$ forms whose cohomology classes are nef. Suppose also that
$$\int_X\alpha^n-n\int_X\alpha^{n-1}\wedge\beta>0.$$
Then we have
\begin{equation}\label{goal3}
\vol(\alpha-\beta)\geq \frac{\left(\int_X\alpha^n-n\int_X\alpha^{n-1}\wedge\beta\right)^n}{\left(\int_X\alpha^n\right)^{n-1}}>0.
\end{equation}
In particular,
\begin{equation}\label{goal4}
\vol(\alpha-\beta)\geq \int_X\alpha^n-n^2\int_X\alpha^{n-1}\wedge\beta.
\end{equation}
\end{theorem}
Here $\vol$ denotes the volume of a $(1,1)$ class, as defined by Boucksom \cite{Bo}, namely $\vol(\alpha)$ is the supremum of $\int_X (T_{ac})^n$ over all closed positive currents $T$ in the class $\alpha$, and $T_{ac}$ denotes the absolutely continuous part of $T$ in its Lebesgue decomposition (and we are setting $\vol(\alpha)=0$ if $\alpha$ is not pseudoeffective).
A conjecture of Boucksom-Demailly-P\u{a}un-Peternell \cite[Conjecture 10.1(ii)]{BDPP} says that in fact we should have
$$\vol(\alpha-\beta)\geq \int_X\alpha^n-n\int_X\alpha^{n-1}\wedge\beta.$$

As a corollary of this result, we obtain a simpler proof of a result of Boucksom  \cite[Theorem 4.1]{Bo}:
\begin{corollary}[Boucksom \cite{Bo}]\label{nefcase}
Let $(X^n,\omega)$ be a compact K\"ahler manifold, and $\alpha$ a closed real $(1,1)$ form whose cohomology class is nef, and with $\int_X\alpha^n>0$. Then
$$\vol(\alpha)=\int_X\alpha^n.$$
\end{corollary}

Indeed, Theorem \ref{main} with $\beta=0$ gives $\vol(\alpha)\geq\int_X\alpha^n$, and the reverse inequality is simple (see \cite[Lemma 4.2]{Bo}).

Another application of this technique is Theorem \ref{effect} below, which we then use to simplify another theorem of Boucksom \cite[Theorem 4.7]{Bo}. We also briefly explain how Theorem \ref{effect} together a recent theorem of Collins and the author \cite{CT} provides another proof of Demailly-P\u{a}un's Nakai-Moishezon criterion for K\"ahler manifolds \cite{DP}, and we use this to characterize the points where the Seshadri constant of a nef class vanishes.

Furthermore, this technique also gives some results in the case when the manifold is not assumed to be K\"ahler. In this case the Calabi-Yau theorem is replaced by its Hermitian counterpart, proved by Weinkove and the author \cite{TW0} (see also \cite{Bl, Che, Chu, DK, GL, KN1, KN2, Nie, TW2, TW5} for earlier results and later developments, \cite{FLY, FWW1, FWW2, FY, LY, Po1, TW3, TW4} for other Monge-Amp\`ere type equations on non-K\"ahler manifolds, and \cite{STW} for a very recent Calabi-Yau theorem for Gauduchon metrics on Hermitian manifolds). The key new difficulty is that now in general we have to modify the function $F$ in \eqref{ma} by adding a constant to it, namely we obtain
\begin{equation}\label{manew}
(\omega+\ddbar\vp)^n=e^{F+b}\omega^n,\quad \omega+\ddbar\vp>0,\quad\sup_X\vp=0,
\end{equation}
for a unique smooth function $\vp$ and unique constant $b\in\mathbb{R}$, and this causes trouble in the mass concentrating procedure. Nevertheless, recent advances towards a conjecture of Demailly-P\u{a}un were recently made by Chiose \cite{Ch} and Nguyen \cite{Ng} along these lines. In Section \ref{sect2} we give an improvement of the main result of \cite{TW1}, which extends a theorem of Demailly \cite{De} to the non-K\"ahler case:

\begin{theorem}\label{main2} Let $X$ be a compact complex $n$-manifold with a closed real $(1,1)$-form $\alpha$ such that its class $[\alpha]\in H^{1,1}_{\mathrm{BC}}(X,\mathbb{R})$ which is nef and big.
Given any $x_1, \dots, x_N \in X$ and  positive real numbers $\tau_1, \dots, \tau_N$ with
\begin{equation}\label{condition1}
\sum_{j=1}^N \tau_j^n < \int_X \alpha^n,
\end{equation}
there exists an $\alpha$-plurisubharmonic function $\vp$ with
\begin{equation}\label{want1}
\varphi(z-x_j) \le \tau_j \log  |z-x_j| + O(1),
\end{equation}
in local coordinates near $x_j$, for all $j=1,\dots,N$.
\end{theorem}

In \cite{TW1} this result was conjectured to hold under the a priori weaker assumption that $[\alpha]$ is nef and $\int_X\alpha^n>0$ (which however conjecturally implies that $[\alpha]$ is big, see Conjecture \ref{dpc} below). The new observation we make is that if we assume that $[\alpha]$ is big then a weak transcendental Morse inequality of Demailly \cite{De2} can be used to control the constant $b$ in \eqref{manew} and make the mass concentration procedure work. We end this note with a discussion of a conjecture of Demailly-P\u{a}un \cite{DP}, and explain some recent results related to this, using the techniques discussed above.

\bigskip
\bigskip
\noindent
{\bf Acknowledgements.}
The author would like to thank Professor S.-T. Yau for many useful conversations, and for his advice and support. He is also grateful to S.Boucksom, T.Collins, N.McCleerey, M.P\u{a}un, G.Sz\'ekelyhidi, B.Weinkove and X.Yang for helpful discussions.

\section{The K\"ahler case}\label{sect1}
In this section we give the proof of Theorem \ref{main}, together with some other applications of the Calabi-Yau theorem to the construction of singular currents on K\"ahler manifolds.

Let us start with some preliminary definitions. If $X$ is a compact complex $n$-manifold, a Hermitian metric $\omega$ on $X$ is called Gauduchon if $\de\db(\omega^{n-1})=0$. Such metrics always exist on any compact complex manifold \cite{Ga}. If $\alpha$ is a smooth closed real $(1,1)$ form on $X$, then $\alpha$ defines a Bott-Chern cohomology class
$[\alpha]\in H^{1,1}_{\mathrm{BC}}(X,\mathbb{R})$ (the space of closed real $(1,1)$ forms modulo $\mn\de\db$-exact ones). We say that the class $[\alpha]$ is nef if for every $\ve>0$ there exists a smooth representative $\alpha+\ddbar\vp_\ve$ of $[\alpha]$ which satisfies $\alpha+\ddbar\vp_\ve\geq -\ve\omega$ on $X$, where $\omega$ is any fixed Hermitian metric.

If $\vp$ is just an upper semicontinuous $L^1$ function such that $\alpha+\ddbar\vp\geq 0$ weakly as currents, then we will say that $\vp$ is $\alpha$-plurisubharmonic. Whenever strict positivity holds, in the sense that $\alpha+\ddbar\vp\geq\ve\omega$ for some $\ve>0$, then we say that $\alpha+\ddbar\vp$ is a K\"ahler current, and the class $[\alpha]$ is called big whenever we can find such a K\"ahler current.

We say that a closed real $(1,1)$ current of the form $T=\alpha+\ddbar\vp$ has analytic singularities if there is some $\delta\in\mathbb{R}_{>0}$ and a coherent ideal sheaf $\mathcal{I}\subset\mathcal{O}_X$ such that given any $x\in X$ there is an open neighborhood $U$ of $x$ and there
are $\{f_1,\dots,f_N\}$ holomorphic functions on $U$, which generate $\mathcal{I}(U)$, such that
$$\vp=\delta\log\left(\sum_{i=1}^N|f_i|^2\right)+\psi,$$
holds on $U$, for some smooth function $\psi$ on $U$. Demailly's regularization theorem \cite{Dem92} shows that every closed real $(1,1)$ current can be approximated by currents with analytic singularities, and with an arbitrarily small loss of positivity.

If $T$ has analytic singularities, and say $T\geq \gamma$ for some smooth form $\gamma$, then the coherent ideal sheaf $\mathcal{I}$ which describes the singularities of $T$ can be principalized thanks to a fundamental theorem of Hironaka \cite{Hi}. This means that there exists a modification $\mu:\ti{X}\to X$, which is a composition of blowups with smooth centers, and with $\ti{X}$ a compact complex manifold, such that $\mu^{-1}\mathcal{I}=\mathcal{O}_{\ti{X}}(-D),$ where $D$ is an effective divisor on $\ti{X}$. If $X$ is K\"ahler, then so is $\ti{X}$. Then we have
$$\mu^*T=\theta+[E],$$
where $\theta$ is a smooth form with $\theta\geq \mu^*\gamma$, and $E$ is an effective $\mathbb{R}$-divisor supported on $D$. We will call this procedure a resolution of the singularities of the current $T$.\\

The starting point for our discussion is the following lemma by Lamari \cite[Lemme 3.3]{La}, whose proof rests on the Hahn-Banach theorem (a technique invented by Sullivan \cite{Su}). For the reader's convenience, we provide the simple proof.

\begin{lemma}[Lamari \cite{La}]\label{lama}
Let $(X^n,\omega)$ be a compact Hermitian manifold, and $\alpha$ a real $(1,1)$ form on $X$. Given $\ve\geq 0$, there exists a current of the form $T=\alpha+\ddbar\vp$ with $T\geq\ve\omega$ if and only if
\begin{equation}\label{inequ}
\int_X\alpha\wedge\chi^{n-1}\geq \ve\int_X \omega\wedge\chi^{n-1},
\end{equation}
holds for all Gauduchon metrics $\chi$ on $X$.
\end{lemma}
\begin{proof}
By replacing $\alpha$ with $\alpha-\ve\omega$, it is enough to consider the case $\ve=0$. If there is a current $T=\alpha+\ddbar\vp\geq 0$ then clearly
$$\int_X \alpha\wedge\chi^{n-1}=\int_X T\wedge\chi^{n-1}\geq 0.$$
Conversely assume that $\int_X\alpha\wedge\chi^{n-1}\geq 0$ for all Gauduchon metrics $\chi$.

If there exists a Gauduchon metric $\chi$ such that
$\int_X\alpha\wedge\chi^{n-1}=0$, then given any $\de\db$-closed real $(n-1,n-1)$ form $\psi$ let
$$f(t)=\int_X \alpha \wedge((1-t)\chi^{n-1}+t\psi).$$
We have $f(0)=0$ and for $|t|$ small we can find a Gauduchon metric $\chi_t$ with $\chi_t^{n-1}=(1-t)\chi^{n-1}+t\psi$ (by simple linear algebra, cf. \cite[(4.8)]{Mi}), hence $f(t)\geq 0$ thanks to our assumption.
But clearly $f(t)$ is a linear function of $t$, and so $f\equiv 0$ and therefore $\int_X\alpha\wedge \psi=0$. Since $\psi$ was an arbitrary $\de\db$-closed real $(n-1,n-1)$ form,
by duality we conclude that $\alpha=\ddbar\vp$ for some distribution $\vp$, and this proves what we want (with $T=0$).

If on the other hand we have $\int_X\alpha\wedge\chi^{n-1}>0$ for all Gauduchon metrics $\chi$, consider the set $H$ of all $\de\db$-closed real $(n-1,n-1)$ forms $\psi$ with $\int_X\alpha\wedge\psi=0$, which is a closed linear subspace of the space of real $(n-1,n-1)$ forms. By assumption have that $H$ is disjoint from the convex set $U$ of strictly positive
real $(n-1,n-1)$ forms, and so we can separate $H$ and $U$ using the Hahn-Banach theorem, and obtain a real $(1,1)$ current $\tau$ which is strictly positive on $U$ and vanishes on $H$. This implies that $\tau\geq 0$. For a given Gauduchon metric $\chi_0$ we have $\int_X\alpha\wedge\chi_0^{n-1}>0$ and $\int_X \tau\wedge\chi_0^{n-1}>0$, so we can write
$$\int_X\alpha\wedge\chi_0^{n-1}=\lambda \int_X \tau\wedge\chi_0^{n-1},$$
where $\lambda>0$. Then $\int_X(\alpha-\lambda\tau)\wedge\psi=0$ for all $\psi\in H$ and also for all $\psi$ which are multiples of $\chi_0$. But $H$ and $\chi_0$ together span the whole space of $\de\db$-closed real $(n-1,n-1)$ forms (since $H$ is a hyperplane in here, and $\chi_0\not\in H$), and so we conclude that
$\int_X(\alpha-\lambda\tau)\wedge\psi=0$ for all $\de\db$-closed real $(n-1,n-1)$ forms. By duality we have that $\alpha-\lambda\tau=\ddbar\vp$ for some distribution $\vp$,
and this again proves what we want (with $T=\alpha-\ddbar\vp$).
\end{proof}

We can now give the proof of Theorem \ref{main}. We will follow closely Popovici \cite{Po}, the only improvement being that we make the arguments effective and obtain explicit constants.
\begin{proof}[Proof of Theorem \ref{main}]
By assumption, for every $\ve>0$ there are functions $f_\ve$ and $g_\ve$ such that
$$\alpha_\ve:=\alpha+\ve\omega+\ddbar f_\ve>0,\quad \beta_\ve:=\beta+\ve\omega+\ddbar g_\ve>0,$$
and note that $[\alpha_\ve-\beta_\ve]=[\alpha-\beta]$. Clearly there exists $\ve_0>0$ such that for all $0\leq \ve< \ve_0$ we have
$$\int_X\alpha_\ve^n-n\int_X\alpha_\ve^{n-1}\wedge\beta_\ve>0.$$
From now on fix any $0<\ve<\ve_0$, and let $\delta=\delta_\ve>0$ be given by
$$\delta=\frac{\int_X\alpha_\ve^n-n\int_X\alpha_\ve^{n-1}\wedge\beta_\ve}{\int_X\alpha_\ve^{n}}.$$
If we can show that given any Gauduchon metric $\chi$ on $X$ we have
\begin{equation}\label{goal}
\int_X (\alpha_\ve-\beta_\ve)\wedge\chi^{n-1}=\int_X (\alpha-\beta)\wedge\chi^{n-1}\geq \delta \int_X \alpha_\ve\wedge\chi^{n-1},
\end{equation}
then Lamari's Lemma \ref{lama} would give us a K\"ahler current $T_\ve\in [\alpha-\beta]$ with $T_\ve\geq \delta\alpha_\ve$.
It follows that
$$T_{\ve, ac}\geq \frac{\int_X\alpha_\ve^n-n\int_X\alpha_\ve^{n-1}\wedge\beta_\ve}{\int_X\alpha_\ve^n}\alpha_\ve,$$
and raising this to the power $n$ and integrating over $X$ we obtain
$$\vol(\alpha-\beta)\geq\frac{\left(\int_X\alpha_\ve^n-n\int_X\alpha_\ve^{n-1}\wedge\beta_\ve\right)^n}{\left(\int_X\alpha_\ve^n\right)^{n-1}},$$
and letting $\ve\to 0$ we obtain \eqref{goal3}.
To get \eqref{goal4}, we just use the elementary inequality $(1-nx)^n\geq 1-n^2x,$ for $ 0\leq x<1/n$.

It remains therefore to show \eqref{goal}, which is equivalent to
$$(1-\delta)\int_X\alpha_\ve\wedge\chi^{n-1}\geq \int_X\beta_\ve\wedge\chi^{n-1},$$
or, substituting the value of $\delta$,
$$\left(\int_X \alpha_\ve\wedge\chi^{n-1}\right) \left(\int_X \beta_\ve\wedge\alpha_\ve^{n-1}\right)\geq\frac{1}{n}\left(\int_X\alpha_\ve^n\right)\left(\int_X\beta_\ve\wedge\chi^{n-1}\right).$$
This follows from the following general fact, observed by Popovici \cite{Po}, which completes the proof of Theorem \ref{main}.
\end{proof}
\begin{lemma}[Popovici \cite{Po}]\label{ineq}
Given $X$ compact K\"ahler $n$-manifold with two K\"ahler metrics $\alpha,\beta$ a Gauduchon metric $\chi$, we have
$$\left(\int_X\alpha\wedge\chi^{n-1}\right)\left(\int_X\beta\wedge\alpha^{n-1}\right)\geq\frac{1}{n}\left(\int_X\alpha^n\right)\left(\int_X \beta\wedge\chi^{n-1}\right).$$\end{lemma}
This inequality is sharp, as simple examples on the torus show.
\begin{proof}[Proof of Lemma \ref{ineq}]
 Use Yau's theorem \cite{Ya} to solve the complex Monge-Amp\`ere equation
$$(\alpha+\ddbar u)^n=\left(\frac{\int_X\alpha^n}{\int_X\beta\wedge\chi^{n-1}}\right)\beta\wedge\chi^{n-1},\quad \ti{\alpha}:=\alpha+\ddbar u>0,$$
where $u$ is a smooth function.
Using the Monge-Amp\`ere equation, we have that
\[\begin{split}
\int_X\ti{\alpha}^n&=\left(\int_X\left(\frac{\ti{\alpha}^n}{\chi^n}\right)^{\frac{1}{2}}\left(\frac{\beta\wedge\chi^{n-1}}{\chi^n}\right)^{\frac{1}{2}}\chi^n
\right)\left(\int_X\alpha^n\right)^{\frac{1}{2}}\left(\int_X\beta\wedge\chi^{n-1}\right)^{-\frac{1}{2}}\\
&=\frac{1}{\sqrt{n}}\left(\int_X\left(\frac{\ti{\alpha}^n}{\chi^n}\right)^{\frac{1}{2}}\left(\tr{\chi}{\beta}\right)^{\frac{1}{2}}\chi^n
\right)\left(\int_X\alpha^n\right)^{\frac{1}{2}}\left(\int_X\beta\wedge\chi^{n-1}\right)^{-\frac{1}{2}},
\end{split}\]
and dividing out and squaring, we get
\[\begin{split}
\int_X\ti{\alpha}^n&=\frac{1}{n}\left(\int_X\left(\frac{\ti{\alpha}^n}{\chi^n}\right)^{\frac{1}{2}}\left(\tr{\chi}{\beta}\right)^{\frac{1}{2}}\chi^n
\right)^2\left(\int_X\beta\wedge\chi^{n-1}\right)^{-1}.
\end{split}\]
Now we have the elementary inequality $\tr{\chi}{\beta}\leq (\tr{\chi}{\ti{\alpha}})(\tr{\ti{\alpha}}{\beta})$, and so
\[\begin{split}
\int_X\ti{\alpha}^n&\leq\frac{1}{n}\left(\int_X\left(\frac{\ti{\alpha}^n}{\chi^n} \tr{\ti{\alpha}}{\beta}\right)^{\frac{1}{2}}\left(\tr{\chi}{\ti{\alpha}}\right)^{\frac{1}{2}}\chi^n
\right)^2\left(\int_X\beta\wedge\chi^{n-1}\right)^{-1}\\
&\leq \frac{1}{n}\left(\int_X (\tr{\ti{\alpha}}{\beta}) \ti{\alpha}^n\right)\left(\int_X (\tr{\chi}{\ti{\alpha}})\chi^n\right)\left(\int_X\beta\wedge\chi^{n-1}\right)^{-1}\\
&=n\left(\int_X \beta\wedge\ti{\alpha}^{n-1}\right)\left(\int_X \ti{\alpha}\wedge\chi^{n-1}\right)\left(\int_X\beta\wedge\chi^{n-1}\right)^{-1},
\end{split}\]
thanks to Cauchy-Schwarz, and we are done.
\end{proof}

Let us also observe the following consequence of the proof of Theorem \ref{main}.

\begin{theorem}\label{effect}
Let $(X^n,\omega)$ be a compact K\"ahler manifold, $\alpha$ a closed real $(1,1)$ form whose cohomology class $[\alpha]$ is nef with $\int_X\alpha^n>0$.
Then there exists a closed positive current $T\in[\alpha]$ with
$$T\geq\frac{\int_X\alpha^n}{n\int_X\alpha^{n-1}\wedge\omega}\omega.$$
The same result holds if $X$ is still K\"ahler but we replace $\omega$ with $\beta$, a closed semipositive real $(1,1)$ form with $\int_X\beta^n>0$.
\end{theorem}
Note that the denominator does not vanish thanks to the Khovanskii-Teissier convexity inequality (see \cite[Proposition 5.2]{De})
$$\int_X\alpha^{n-1}\wedge\omega\geq \left(\int_X\alpha^n\right)^{\frac{n-1}{n}}\left(\int_X\omega^n\right)>0.$$
\begin{proof}
It is enough to prove the more general statement with $\beta$ semipositive with positive volume.
Replacing $[\alpha]$ by $[\alpha+\ve\omega], \ve>0$, and $\beta$ by $\beta+\ve\omega$, and then taking $\ve\to 0$, we are reduced to the case when $[\alpha]$ is a K\"ahler and $\beta$ is a K\"ahler form.
Thanks to Lamari's Lemma \ref{lama}, all we need to show is that for any Gauduchon metric $\chi$ on $X$ we have
$$\int_X\alpha\wedge\chi^{n-1}\geq \frac{\int_X\alpha^n}{n\int_X\alpha^{n-1}\wedge\beta}\int_X\beta\wedge\chi^{n-1},$$
and this is precisely the inequality proved in Lemma \ref{ineq}.
\end{proof}

As a corollary, we prove the following theorem of Boucksom \cite[Theorem 4.7]{Bo}. The proof is the same as Boucksom's, except that we replace the complicated mass concentration arguments there (which are generalizations of the ones in \cite{DP}) with the simpler Theorem \ref{effect}.

\begin{theorem}\label{vol}
Let $(X^n,\omega)$ be a compact K\"ahler manifold, and $\alpha$ a closed real $(1,1)$ form. Then the class $[\alpha]$ is big iff $\vol(\alpha)>0$.
\end{theorem}
\begin{proof}
If $[\alpha]$ is big then it contains a K\"ahler current $T\geq \ve\omega$ for some $\ve>0$. Then $T_{ac}\geq \ve\omega$ as well, and so
$$\vol(\alpha)\geq\int_X(T_{ac})^n\geq\ve^n\int_X\omega^n>0.$$
Assume conversely that $\vol(\alpha)>0$, so by definition there is a  closed positive current $S$ in the class $[\alpha]$ with
$$\int_X(S_{ac})^n=\frac{\vol(\alpha)}{2}>0.$$
Applying Demailly's regularization (see \cite{Dem92} or \cite[Theorem 2.4]{Bo}) to $S$ we obtain closed positive currents $T_{k}\in[\alpha]$ with analytic singularities, with $T_{k}\geq -\ve_k\omega$, $\ve_k\to 0$, and with $(T_{k})_{ac}\to S_{ac}$ a.e. as $k\to\infty$. Thanks to Fatou's lemma,
$$\liminf_{k\to\infty}\int_X\left(\left(T_{k}+\ve_k\omega\right)_{ac}\right)^n\geq \int_X(S_{ac})^n,$$
and so for all $k$ large we have
$$T_k\geq -\ve_k\omega,\quad \int_X ((T_k+\ve_k\omega)_{ac})^n\geq \eta>0,$$
for some $\eta>0$. Let $\mu_k:X_k\to X$ be a resolution of the singularities of $T_k$, so that $\mu_k$ is a composition of blowups with smooth centers (and so $X_k$ is K\"ahler) and we have
$$\mu_k^*T_k=\theta_k+[E_k],$$
where $\theta_k$ is a smooth closed real $(1,1)$ form with $\theta_k\geq -\ve_k\mu_k^*\omega$, and $E_k$ is an effective $\mathbb{R}$-divisor. We have
$$\int_X((T_k+\ve_k\omega)_{ac})^n=\int_{X_k} ((\mu_k^*T_k+\ve_k\mu_k^*\omega)_{ac})^n=\int_{X_k} (\theta_k+\ve_k \mu_k^*\omega)^n\geq \eta>0,$$
and so the nef class $[\theta_k+\ve_k\mu_k^*\omega]$ on $X_k$ satisfies the hypotheses of Theorem \ref{effect}, where we take $\beta=\mu_k^*\omega$.
We obtain a closed positive current $\ti{T}_k$ on $X_k$ in the class $[\theta_k+\ve_k\mu_k^*\omega]$ satisfying
$$\ti{T}_{k}\geq \frac{\int_{X_k}(\theta_k+\ve_k \mu_k^*\omega)^n}{n\int_{X_k}(\theta_k+\ve_k \mu_k^*\omega)^{n-1}\wedge \mu_k^*\omega} \mu_k^*\omega.$$
As remarked above we have $\int_{X_k}(\theta_k+\ve_k \mu_k^*\omega)^n\geq \eta>0,$ and we also have
$$\int_{X_k}(\theta_k+\ve_k \mu_k^*\omega)^{n-1}\wedge \mu_k^*\omega=\int_X
((T_k+\ve_k\omega)_{ac})^{n-1}\wedge\omega,$$
which is bounded above independent of $k$ (see \cite[Proposition 2.6]{Bo}). It follows that there is a constant $\eta'>0$ such that for all $k$ large we have
$$\ti{T}_{k}\geq \eta'\mu_k^*\omega,$$
and so for $k$ large we obtain
$$(\mu_k)_*(\ti{T}_{k}+[E_k])-\ve_k\omega\geq \frac{\eta'}{2}\omega,$$
which is a K\"ahler current in the class $[\alpha]$.
\end{proof}

Let now $X$ be a compact complex manifold and $[\alpha]$ a big class, i.e. it contains a K\"ahler current. By Demailly's regularization \cite{Dem92} the class $[\alpha]$ contains a K\"ahler current with analytic singularities, and therefore if we define the non-K\"ahler locus $E_{nK}(\alpha)$ to be the intersection of the singular loci of all K\"ahler currents with analytic singularities in the class $[\alpha]$ then we see that $E_{nK}(\alpha)$ is a proper analytic subvariety of $X$. If the class $[\alpha]$ is not big, then we just set $E_{nK}(\alpha)=X$.

The main theorem of \cite{CT}, which reproves and generalizes algebro-geometric results by Nakamaye \cite{Nak} and Ein-Lazarsfeld-Musta\c{t}\u{a}-Nakamaye-Popa \cite{ELMNP}, then can be stated as follows:
\begin{theorem}[Collins-Tosatti \cite{CT}]\label{null} Let $X$ be a compact complex manifold and $\alpha$ a closed real $(1,1)$ form whose cohomology class $[\alpha]$ is nef and big. Then we have
$$E_{nK}(\alpha)=\bigcup_{\int_V\alpha^{\dim V}=0}V,$$
where the union is over all irreducible analytic subvarieties $V\subset X$ with $\int_V\alpha^{\dim V}=0$.
\end{theorem}
The set on the right hand side is called the null locus of the class $[\alpha]$. Note that the existence of a big class on $X$ implies that $X$ is in Fujiki's class $\mathcal{C}$ (i.e. bimeromorphic to a compact K\"ahler manifold), thanks to \cite[Theorem 3.4]{DP}, and so this theorem  follows from \cite[Theorem 1.1]{CT}. The proof in \cite{CT} uses the mass concentration technique of \cite{DP}, but this can now be replaced by Theorem \ref{effect}, and therefore the proof of Theorem \ref{null} can be made independent of the main results of \cite{DP}.

As an application, we discuss the following Nakai-Moishezon criterion for K\"ahler manifolds due to Demailly-P\u{a}un \cite{DP}:
\begin{theorem}\label{dpt}
Let $(X^n,\omega)$ be a compact K\"ahler manifold, and $\alpha$ a closed real $(1,1)$ form whose cohomology class $[\alpha]$ is nef.
Then $[\alpha]$ is a K\"ahler class iff $\int_V\alpha^{\dim V}>0$ for all irreducible analytic subvarieties $V\subset X$.
\end{theorem}

\begin{proof}
Clearly if $[\alpha]$ is K\"ahler then  $\int_V\alpha^{\dim V}>0$ for all such $V$. For the other implication, taking $V=X$ we see that $\int_X\alpha^n>0$, and so Theorem \ref{effect} implies that the class $[\alpha]$ is big, i.e. it contains a K\"ahler current.
Applying Theorem \ref{null} we obtain
$$E_{nK}(\alpha)=\bigcup_{\int_V\alpha^{\dim V}=0}V,$$
and by assumption we have that the null locus of $[\alpha]$ is empty, and therefore so is $E_{nK}(\alpha)$.
But Boucksom \cite[Theorem 3.17]{Bo2} observed that we can find a K\"ahler current $T$ with analytic singularities in the class $[\alpha]$ which is singular precisely along $E_{nK}(\alpha)$, and so the K\"ahler current $T$ is in fact a smooth K\"ahler metric.
\end{proof}

We also give an application of Theorem \ref{null} to Seshadri constants. These were introduced by Demailly \cite{De3} as a way to measure local positivity properties of line bundles, and we refer the reader to \cite{De3, Laz} for more properties of these invariants.

Let $X$ be a compact K\"ahler manifold and $\alpha$ a closed real $(1,1)$ form whose cohomology class $[\alpha]$ is nef. Given any $x\in X$
we let the Seshadri constant of $\alpha$ at $x$ be
$$\ve(\alpha,x)=\sup\{ \lambda\geq 0\ |\ \pi^*[\alpha]-\lambda [E]\ {\rm nef}\},$$
where $\pi:\ti{X}\to X$ is the blowup of $X$ at $x$, and $E$ is the exceptional divisor. It is easy to see that $\ve(\cdot, x)$ is continuous on the nef cone.

The following result is an application of Theorem \ref{null}:
\begin{theorem}\label{ses}
Let $X$ be a compact K\"ahler manifold, $\alpha$ a closed real $(1,1)$ form whose cohomology class $[\alpha]$ is nef, and $x$ be a point in $X$. Then
$$\ve(\alpha,x)=0\quad\Longleftrightarrow\quad x\in E_{nK}(\alpha).$$
\end{theorem}
\begin{proof}
If $x\in E_{nK}(\alpha)$ then by Theorem \ref{null} there is an irreducible $k$-dimensional analytic subvariety $V$ which passes through $x$ and with
$$\int_V\alpha^k=0.$$
If $\ti{V}$ denotes the proper transform of $V$ under $\pi$, then for any $\lambda>0$ we have
$$\int_{\ti{V}}(\pi^*\alpha-\lambda[E])^k=(-1)^k\lambda^k\int_{\ti{V}}[E]^k=-\lambda^k\mathrm{mult}_xV<0,$$
where the last equality is \cite[Lemma 5.1.10]{Laz}, and so $\pi^*[\alpha]-\lambda [E]$ is not nef, thus showing that $\ve(\alpha,x)=0$.

On the other hand if $x\not\in E_{nK}(\alpha)$ then there is a K\"ahler current $T\in[\alpha]$ with analytic singularities which is smooth near $x$, and satisfies $T=\alpha+\ddbar\psi\geq\beta\omega$ on $X$ for some $\beta>0$. Also, since $[\alpha]$ is nef, for every $\ve>0$ there is a smooth function $\rho_\ve$ such that
$\alpha+\ddbar\rho_\ve\geq -\ve\omega$ on $X$. Now since $[E]|_{E}\cong\mathcal{O}_{\mathbb{CP}^{n-1}}(-1)$, there is a smooth closed real $(1,1)$ form $\eta$ on $\ti{X}$ which is cohomologous to $[E]$, is supported on a neighborhood $U$ of $E$, and such that $\pi^*\omega-\gamma\eta$ is a K\"ahler metric on $\ti{X}$ for some small $\gamma>0$ (see e.g. \cite[Lemma 3.5]{DP}).
If $\widetilde{\max}$ denotes a regularized maximum function (see \cite[I.5.18]{Demb}), then
for every $\ve>0$ we can choose a large constant $C_\ve>0$ such that
$$\ti{\rho}_\ve:=\widetilde{\max}(\rho_\ve-C_\ve,\psi),$$
is a smooth function on $X$ which agrees with $\psi$ in a neighborhood of $x$ which contains $\pi(U)$. We have $\alpha+\ddbar\ti{\rho}_\ve\geq -\ve\omega$ on $X$ and $\alpha+\ddbar\ti{\rho}_\ve\geq \beta\omega$ on $\pi(U)$. Therefore on $U$ we have
$$\pi^*\alpha-\beta\gamma\eta+\ddbar(\pi^*\ti{\rho}_\ve)>0,$$
while on $\ti{X}\backslash U$ we have
$$\pi^*\alpha-\beta\gamma\eta+\ddbar(\pi^*\ti{\rho}_\ve)=\pi^*\alpha+\ddbar(\pi^*\ti{\rho}_\ve)\geq-\ve\pi^*\omega\geq -C\ve\ti{\omega},$$
where $\ti{\omega}$ is a fixed K\"ahler metric on $\ti{X}$, with $\pi^*\omega\leq C\ti{\omega}$.
This shows that $\pi^*[\alpha]-\beta\gamma[E]$ is nef, and so $\ve(\alpha,x)\geq \gamma\beta>0$.
\end{proof}

In fact, we can sharpen this result further.
The following is the analog of \cite[Proposition 5.1.9]{Laz} in our transcendental situation.
\begin{theorem}\label{subvar}
With the same assumptions as in Theorem \ref{ses}, we have
\begin{equation}\label{goall}
\ve(\alpha,x)=\inf_{V\ni x}\left(\frac{\int_{V}\alpha^{\dim V}}{\mathrm{mult}_x V}\right)^{\frac{1}{\dim V}},
\end{equation}
where the infimum is over all irreducible analytic subvarieties $V$ containing $x$. In particular, $\ve(\alpha,x)\leq \vol(\alpha)^{\frac{1}{n}}$. Lastly, if $[\alpha]$ is K\"ahler then the infimum in \eqref{goall} is actually a minimum.
\end{theorem}
\begin{proof}
Let $V$ be any irreducible subvariety through $x$, let $\ti{V}$ be the proper transform of $V$, and let $k=\dim V$. If $\pi^*[\alpha]-\lambda[E]$ is nef then
$$0\leq \int_{\ti{V}}(\pi^*\alpha-\lambda[E])^{k}=\int_V\alpha^k+(-1)^k\lambda^k\int_{\ti{V}}[E]^k=\int_V\alpha^k-\lambda^k\mathrm{mult}_x V.$$
This proves that
\begin{equation}\label{str2}
\ve(\alpha,x)\leq\inf_{V\ni x}\left(\frac{\int_{V}\alpha^{\dim V}}{\mathrm{mult}_x V}\right)^{\frac{1}{\dim V}},
\end{equation}
for all nef classes $[\alpha]$. Next, we show that \eqref{goall} holds
when $[\alpha]$ is K\"ahler. In this case we have $\ve(\alpha,x)>0$ and the class $\pi^*[\alpha]-\ve(\alpha,x)[E]$ is nef but not K\"ahler, so by Theorem \ref{dpt} there is an irreducible analytic subvariety $\ti{V}\subset \ti{X}$ with
\begin{equation}\label{str}
0=\int_{\ti{V}}(\pi^*\alpha-\ve(\alpha,x)[E])^{k}.
\end{equation}
If $\ti{V}$ were disjoint from $E$ then $V=\pi(\ti{V})$ would be an irreducible $k$-dimensional subvariety of $X$ and we would have
$$0=\int_{\ti{V}}(\pi^*\alpha)^k=\int_V\alpha^k,$$
a contradiction since $[\alpha]$ is K\"ahler.

If we had $\ti{V}\subset E$ then we would have $-[E]|_{\ti{V}}\cong \mathcal{O}_{\mathbb{CP}^{n-1}}(1)|_{\ti{V}}$ which is ample, and so
$$0=\ve(\alpha,x)^k\int_{\ti{V}}(-[E])^{k}>0,$$
a contradiction.

Therefore we must have that $\ti{V}$ intersects $E$ but is not contained in it, and so $V=\pi(\ti{V})$ is an irreducible $k$-dimensional subvariety of $X$ which passes through $x$ and \eqref{str} gives
$$\ve(\alpha,x)=\left(\frac{\int_{V}\alpha^{k}}{\mathrm{mult}_x V}\right)^{\frac{1}{k}},$$
which together with \eqref{str2} proves \eqref{goall} when $[\alpha]$ is K\"ahler, with the infimum being in fact a minimum.

Next, for a general nef class $[\alpha]$ and any $x\in X$, the class $[\alpha+\ve\omega]$ is K\"ahler for every $\ve>0$. Using the continuity of the Seshadri constants, together with \eqref{goall} for K\"ahler classes, we obtain
\[\begin{split}
\ve(\alpha,x)&=\lim_{\ve\downarrow 0}\ve(\alpha+\ve\omega,x)=\lim_{\ve\downarrow 0}\min_{V\ni x}\left(\frac{\int_{V}(\alpha+\ve\omega)^{\dim V}}{\mathrm{mult}_x V}\right)^{\frac{1}{\dim V}}\\
&\geq \inf_{V\ni x}\left(\frac{\int_{V}\alpha^{\dim V}}{\mathrm{mult}_x V}\right)^{\frac{1}{\dim V}},\end{split}\]
which combined with \eqref{str2} proved \eqref{goall}.
Lastly, the bound $\ve(\alpha,x)\leq \vol(\alpha)^{\frac{1}{n}}$ follows from Corollary \ref{nefcase}.
\end{proof}

\section{The non-K\"ahler case}\label{sect2}

In this section we consider the general case of compact complex manifolds which are not assumed to be K\"ahler.

In his seminal paper \cite{De}, Demailly used a mass concentration technique for a degenerating family of complex Monge-Amp\`ere equations to produce $\alpha$-plurisubharmonic functions in a nef and big class $[\alpha]$ (on a compact K\"ahler manifold) which have nontrivial Lelong numbers (i.e. logarithmic poles singularities) at finitely many given points. The precise statement that he
proved is exactly the same as our Theorem \ref{main2} but with the additional assumption that $X$ be K\"ahler.
This was later extended by Weinkove and the author \cite{TW1} to certain non-K\"ahler manifolds. More precisely, we proved the statement in Theorem \ref{main2}, replacing the assumption that $[\alpha]$ is big with the weaker assumption that $\int_X\alpha^n>0$, but we also had to assume that either $n\leq 3$ or that $X$ is Moishezon and the class $[\alpha]$ is rational.
More recently Nguyen \cite{Ng} proved Theorem \ref{main2} assuming that $\int_X\alpha^n>0$ and $\alpha$ is semipositive definite (which is in general stronger than assuming that $[\alpha]$ is nef).

Our new observation is that if we assume that $[\alpha]$ is big then a weak transcendental Morse inequality of Demailly \cite[Theorem 1.18(a)]{De2} can be brought to bear. More precisely, we use the following result, which is reminescent of Y.T. Siu's ``calculus inequalities'' \cite{Si}, and which relies crucially on the regularity results for envelopes proved in \cite{BD}.
\begin{theorem}[Demailly \cite{De2}]\label{con}
Let $X$ be a compact complex $n$-manifold with a closed real $(1,1)$-form $\alpha$ such that its class $[\alpha]\in H^{1,1}_{\mathrm{BC}}(X,\mathbb{R})$ which is nef and big, and let
$\beta=\alpha+\ddbar f$ be any smooth representative of $[\alpha]$. Then
$$\int_X \alpha^n=\int_X\beta^n\leq\int_{X(\beta,0)}\beta^n,$$
where $X(\beta,0)$ is the subset of $X$ where $\beta$ is semipositive definite.
\end{theorem}
\begin{proof}
Since $[\alpha]$ is big, we have that the manifold $X$ is in class $\mathcal{C}$, by \cite[Theorem 3.4]{DP}, and so there is a modification $\mu:\ti{X}\to X$ with $\ti{X}$ a compact K\"ahler manifold. If $T$ is a K\"ahler current on $X$ in the class $[\alpha]$, then as in \cite[Lemma 3.5]{DP} we have that $\mu^*T+\ddbar h$ is a K\"ahler current on $\ti{X}$ in the class $[\mu^*\alpha]$ for a suitable $L^1$ function $h$ (singular only along $\mathrm{Exc}(\mu)$). The class $[\mu^*\alpha]$ is therefore nef and big, hence
$$0<\vol(\mu^*\alpha)=\int_{\ti{X}}(\mu^*\alpha)^n=\int_X\alpha^n,$$
 by Corollary \ref{nefcase}.
Then \cite[Theorem 1.18(a)]{De2} gives
$$\vol(\mu^*\alpha)\leq \inf_{\gamma \in [\mu^*\alpha]}\int_{\ti{X}(\gamma,0)}\gamma^n\leq \int_{X(\beta,0)}\beta^n,$$
where the infimum is over all smooth closed real $(1,1)$ forms $\gamma$ in the class $[\mu^*\alpha],$ and the last inequality follows by taking $\gamma=\mu^*\beta$.
\end{proof}

\begin{remark}
It is a very interesting problem to show that Theorem \ref{con} remains true if only assume that $[\alpha]$ is nef and $\int_X\alpha^n>0$. This is certainly expected to be true, thanks to Conjecture \ref{dpc} below, and conversely solving this problem would be helpful to attack this conjecture.
\end{remark}

\begin{proof}[Proof of Theorem \ref{main2}]
We follow the arguments of Demailly \cite[Section 6]{De} in the K\"ahler case, which were extended in \cite{TW1} to the non-K\"ahler case (under extra hypotheses). We will be brief wherever the arguments are the same as in \cite{TW1}, and we just will explain in detail the new ingredients needed here.

As in \cite{De, TW1} for each $j=1,\dots,N$ and each $\ve>0$ we construct a smooth nonnegative $(n,n)$-form $\gamma_{j,\ve}^n$ on $X$ with $\int_X\gamma_{j,\ve}^n=1$ and $\gamma_{j\ve}^n\rightharpoonup \delta_{x_j}$ as $\ve\to 0$.
Since $[\alpha]$ is nef, for every small $\ve>0$ we can find a smooth function $\psi_\ve$ such that $\alpha+\ve\omega+\ddbar\psi_\ve>0$, where $\omega$ is a fixed Hermitian metric on $X$. Choose also $\delta>0$ small enough so that $$\sum_{j=1}^N\tau_j^n+\delta\int_X\omega^n<\int_X\alpha^n.$$
Using the solution of the complex Monge-Amp\`ere equation on compact Hermitian manifolds \cite{TW0}, for every small $\ve>0$ we obtain a smooth function $\vp_\ve$ and a constant $C_\ve$ such that
$$\left(\alpha+\ve\omega+\ddbar(\psi_\ve+\vp_\ve)\right)^n=C_\ve\left(\sum_{j=1}^N\tau_j^n\gamma_{j,\ve}^n+\delta\omega^n\right),$$
with $$\alpha+\ve\omega+\ddbar(\psi_\ve+\vp_\ve)>0,\quad \sup_X(\psi_\ve+\vp_\ve)=0.$$
If we can show that
$$C_\ve\geq 1,$$
then the rest of the proof follows exactly as in \cite{TW1}.
To see this, let $\beta=\alpha+\ddbar(\psi_\ve+\vp_\ve)$, which is a smooth representative of $[\alpha]$. Thanks to Theorem \ref{con}, we have
\[\begin{split}
\int_X \alpha^n&\leq \int_{X(\beta,0)}(\alpha+\ddbar(\psi_\ve+\vp_\ve))^n\\
&\leq \int_{X(\beta,0)}\left(\alpha+\ve\omega+\ddbar(\psi_\ve+\vp_\ve)\right)^n\\
&=C_\ve\int_{X(\beta,0)}\left(\sum_{j=1}^N\tau_j^n\gamma_{j,\ve}^n+\delta\omega^n\right)\\
&\leq C_\ve\int_X\left(\sum_{j=1}^N\tau_j^n\gamma_{j,\ve}^n+\delta\omega^n\right)\\
&=C_\ve \left(\sum_{j=1}^N \tau_j^n+\delta\int_X\omega^n\right)<C_\ve\int_X\alpha^n,
\end{split}
\]
as required.
\end{proof}
\begin{remark}
More generally, the same argument shows that if, in the same setup as above, we solve the complex Monge-Amp\`ere equations
$$\left(\alpha+\ve\omega+\ddbar(\psi_\ve+\vp_\ve)\right)^n=C_\ve\Omega_\ve,$$
with $$\alpha+\ve\omega+\ddbar(\psi_\ve+\vp_\ve)>0,\quad \sup_X(\psi_\ve+\vp_\ve)=0,$$
where for every $\ve>0$, $\Omega_\ve$ is a smooth positive volume form with $\int_X\Omega_\ve=1$, then we have
$$C_\ve\geq \int_X\alpha^n.$$
\end{remark}

Next, we consider a well-known conjecture of Demailly-P\u{a}un \cite{DP}:
\begin{conjecture}[Demailly-P\u{a}un \cite{DP}]\label{dpc}
Let $X$ be a compact complex $n$-fold and $\alpha$ a closed real $(1,1)$ form such that its class $[\alpha]\in H^{1,1}_{\mathrm{BC}}(X,\mathbb{R})$ is nef and $\int_X\alpha^n>0$. Then $[\alpha]$ is big.
\end{conjecture}

When $X$ is K\"ahler, this was proved in \cite{DP} using the mass concentration technique. Note that this is now a direct consequence of Theorem \ref{effect}.

When $n=2$ the manifold $X$ as in Conjecture \ref{dpc} is necessarily K\"ahler (see e.g. \cite[p.4005]{TW1}), and so the conjecture holds. Very recently, Chiose \cite{Ch} has proved the following special case of this conjecture:

\begin{theorem}[Chiose \cite{Ch}]
Conjecture \ref{dpc} holds if $X$ admits a Hermitian metric $\omega$ with $\de\db\omega=0=\de\db(\omega^2)$.
\end{theorem}

We now give a proof of this result using the method of Popovici \cite{Po}.
\begin{proof}
By assumption, for every $\ve>0$ there is a smooth function $\psi_\ve$ such that
$$\alpha_\ve:=\alpha+\ve\omega+\ddbar \psi_\ve>0.$$
If we can show that there exist $\ve,\delta>0$ such that for every Gauduchon metric $\chi$ on $X$ we have
\begin{equation}\label{want}
\int_X \alpha_\ve\wedge\chi^{n-1}\geq \ve(1+\delta)\int_X \omega\wedge\chi^{n-1},
\end{equation}
then Lamari's Lemma \ref{lama} would give us a distribution $\vp$ such that
$$\alpha_\ve+\ddbar\vp \geq \ve(1+\delta)\omega,$$
and so
$$\alpha+\ddbar(\psi_\ve+\vp)\geq\delta\ve\omega,$$
and we would be done. To prove \eqref{want}, given any $\ve>0$ and $\chi$ as above we solve the complex Monge-Amp\`ere equation \cite{TW0}
$$(\alpha_\ve+\ddbar \vp_\ve)^n=e^{b_\ve}\frac{\omega\wedge\chi^{n-1}}{\int_X\omega\wedge\chi^{n-1}},\quad \ti{\alpha}_\ve:=\alpha_\ve+\ddbar \vp_\ve>0,$$
where $\vp_\ve$ is a smooth function and $b_\ve\in\mathbb{R}$. We have
\[\begin{split}e^{b_\ve}&=\int_X \ti{\alpha}_\ve^n=e^{\frac{b_\ve}{2}}\left(\int_X \left(\frac{\ti{\alpha}_\ve^n}{\chi^n}\right)^{\frac{1}{2}}\left(\frac{\omega\wedge\chi^{n-1}}{\chi^n}\right)^{\frac{1}{2}}
\chi^n\right)\left(\int_X\omega\wedge\chi^{n-1}\right)^{-\frac{1}{2}}\\
&=\frac{e^{\frac{b_\ve}{2}}}{\sqrt{n}}\left(\int_X \left(\frac{\ti{\alpha}_\ve^n}{\chi^n}\right)^{\frac{1}{2}}\left(\tr{\chi}{\omega}\right)^{\frac{1}{2}}
\chi^n\right)\left(\int_X\omega\wedge\chi^{n-1}\right)^{-\frac{1}{2}},\end{split}\]
and dividing out and squaring
\[\begin{split}e^{b_\ve}&=\frac{1}{n}\left(\int_X \left(\frac{\ti{\alpha}_\ve^n}{\chi^n}\right)^{\frac{1}{2}}\left(\tr{\chi}{\omega}\right)^{\frac{1}{2}}
\chi^n\right)^2\left(\int_X\omega\wedge\chi^{n-1}\right)^{-1}\\
&\leq \frac{1}{n}\left(\int_X \left(\frac{\ti{\alpha}_\ve^n}{\chi^n}\tr{\ti{\alpha}_\ve}{\omega}\right)^{\frac{1}{2}}\left(\tr{\chi}{\ti{\alpha}_\ve}\right)^{\frac{1}{2}}
\chi^n\right)^2\left(\int_X\omega\wedge\chi^{n-1}\right)^{-1}\\
&\leq \frac{1}{n}\left(\int_X (\tr{\ti{\alpha}_\ve}{\omega})\ti{\alpha}_\ve^n\right)\left(\int_X (\tr{\chi}{\ti{\alpha}_\ve})\chi^n\right)\left(\int_X\omega\wedge\chi^{n-1}\right)^{-1}\\
&=n\left(\int_X \omega\wedge\ti{\alpha}_\ve^{n-1}\right)\left(\int_X\ti{\alpha}_\ve\wedge\chi^{n-1}\right)\left(\int_X\omega\wedge\chi^{n-1}\right)^{-1},\end{split}\]
i.e.
$$\frac{\int_X\alpha_\ve\wedge\chi^{n-1}}{\ve\int_X\omega\wedge\chi^{n-1}}\geq \frac{\int_X\ti{\alpha}_\ve^n}{n\ve\int_X\omega\wedge\ti{\alpha}_\ve^{n-1}}.$$
Therefore to prove \eqref{want} it is enough to show that there exist $\ve,\delta'>0$, independent of $\chi$, such that
\begin{equation}\label{want2}
n\ve\int_X\omega\wedge\ti{\alpha}_\ve^{n-1}\leq (1-\delta')\int_X\ti{\alpha}_\ve^n.
\end{equation}
We now use the assumption that $\de\db\omega=0=\de\db(\omega^2)$. An elementary calculation shows that this implies that $\de\db(\omega^k)=0$ for all $1\leq k\leq n-1$. Therefore for all
$1\leq k\leq n$ we have
$$\int_X \ti{\alpha}_\ve^k\wedge\omega^{n-k}=\int_X\alpha^k\wedge\omega^{n-k},$$
and \eqref{want2} reduces to proving
$$n\ve\int_X\omega\wedge\alpha^{n-1}\leq (1-\delta')\int_X\alpha^n,$$
which is obvious when we take $\ve$ small enough.
\end{proof}

As this proof shows, to solve Conjecture \ref{dpc} in general, it would suffice to show that \eqref{want2} holds. This seems very hard to prove, without assuming that $\de\db\omega=0=\de\db(\omega^2)$. However, when $n=3$, the weaker inequality
\begin{equation}\label{want3}
3\ve\int_X\omega\wedge\ti{\alpha}_\ve^{2}< \int_X\ti{\alpha}_\ve^3,
\end{equation}
holds in general, as long as $\ve$ is small enough (independent of $\chi$), and provided we choose $\omega$ Gauduchon. In other words, we can prove \eqref{want2}, but the constant $\delta'$ may depend on $\chi$. Indeed, if we set $u_\ve:=\psi_\ve+\vp_\ve$ and expand
\[\begin{split}
3\ve\int_X\omega\wedge\ti{\alpha}_\ve^2&=3\ve\int_X\omega\wedge(\alpha+\ddbar u_\ve)^2+6\ve^2\int_X\omega^2\wedge(\alpha+\ddbar u_\ve)\\
&\ \ \ \ +3\ve^3\int_X\omega^3\\
&=3\ve\int_X\omega\wedge(\alpha+\ddbar u_\ve)^2+6\ve^2\int_X\omega^2\wedge\alpha+3\ve^3\int_X\omega^3,
\end{split}\]
using that $\omega$ is Gauduchon, while
\[\begin{split}
\int_X\ti{\alpha}_\ve^3&=\int_X(\alpha+\ddbar u_\ve)^3+3\ve\int_X (\alpha+\ddbar u_\ve)^2\wedge\omega\\
&\ \ \ \ +
3\ve^2\int_X (\alpha+\ddbar u_\ve)\wedge\omega^2+\ve^3\int_X\omega^3\\
&=\int_X\alpha^3+3\ve\int_X (\alpha+\ddbar u_\ve)^2\wedge\omega+3\ve^2\int_X \alpha\wedge\omega^2+\ve^3\int_X\omega^3,
\end{split}\]
and so
\[\begin{split}
\int_X\ti{\alpha}_\ve^3-3\ve\int_X\omega\wedge\ti{\alpha}_\ve^2&=
\int_X\alpha^3-3\ve^2\int_X \alpha\wedge\omega^2-2\ve^3\int_X\omega^3>0,
\end{split}\]
as long as $\ve$ is small enough, which proves \eqref{want3}.

Note that to promote this to \eqref{want2}, it would be enough to show that $\int_X\ti{\alpha}_\ve^3\leq C$, independent of $\ve$, or equivalently
$\ve\int_X\ti{\alpha}_\ve^2\wedge\omega\leq C$.

Furthermore, these calculations also show that when $n=3$ Conjecture \ref{dpc} holds if we just assume that $\de\db\omega=0$, a fact which was observed by Chiose \cite{Ch}.

Lastly, we mention that very recently Nguyen \cite{Ng} has proved Conjecture \ref{dpc} assuming that the form $\alpha$ is semipositive definite, and that there exists a Hermitian metric $\omega$ with $\de\db\omega=0$.

\end{document}